\RequirePackage{amsmath}
\documentclass[runningheads]{llncs}

\usepackage[T1]{fontenc}

\usepackage{amsfonts, mathtools, bm, amssymb, mathrsfs}
\usepackage{algorithm, algorithmic, setspace}

\usepackage{graphicx}
\usepackage{subfigure}

\usepackage{hyperref}
\usepackage{color}

\def\N			{\mathbb N}
\def\Z			{\mathbb Z}
\def\R			{\mathbb R}

\def\Ball		{\mathbb B}

\def\Radon		{\mathcal R}
\def\Back		{\mathcal B}
\def\Fourier		{\mathcal F}
\def\Int 		{\mathfrak I}

\def\Cont		{\mathrm C}
\def\Lebesgue	{\mathrm L}
\def\Sobolev		{\mathrm H}
\def\Bernstein	{\mathscr B}

\def\d			{\mathrm d}
\def\e			{\mathrm e}
\def\i			{\mathrm i}
\def\T			{\mathrm T}

\def\Mod			{\mathcal M}

\DeclareMathOperator{\DFT}{DFT}
\DeclareMathOperator{\IDFT}{iDFT}
\DeclareMathOperator{\SSIM}{SSIM}
\DeclareMathOperator{\supp}{supp}
\DeclareMathOperator{\round}{round}
\DeclareMathOperator{\per}{per}
\DeclareMathOperator{\FBP}{FBP}
\DeclareMathOperator{\LMU}{LMU}

\newcommand{\norm}[2]{\left\lVert#1\right\rVert_{{#2}}}
\newcommand{\abs}[1]{\left\lvert#1\right\rvert}
\newcommand{\floor}[1]{\ensuremath{\left \lfloor {#1} \right \rfloor}}
\newcommand{\closure}[1]{\ensuremath{\overline{{#1}}}}

\newcommand{\Lper}{\Lebesgue^2_{\per}}
\newcommand{\HperT}[1]{\ensuremath{\Sobolev^{1}_{\per}}(\T^{{#1}})}
\newcommand{\Hper}{\ensuremath{\Sobolev^{1}_{\per}}}
\newcommand{\HperDT}[1]{\ensuremath{\Sobolev^{1}_{\per,D}}(\T^{{#1}})}
\newcommand{\cperD}[2]{\ensuremath{\Cont_{\per, D}}^{{#1}}(\T^{{#2}})}

\newcommand{\laplace}{\Delta}
\newcommand{\deriv}[2][]{\ensuremath{\frac{\partial^{#1} }{\partial {#2}^{#1}}}}

\newcommand{\set}[2]{\ensuremath{\left\{ {#1}  \, \middle \vert \,   {#2} \right\}}}
\newcommand{\idxset}[2]{\ensuremath{\left\{{#1}, \dots, {#2}\right\}}}
\newcommand{\bandw}{\ensuremath{L}}

\begin{document}
\title{On an Analytical Inversion Formula for the Modulo Radon Transform}
\titlerunning{Analytical Inversion Formula for the Modulo Radon Transform}
\author{Matthias Beckmann\inst{1,2} \and Carla Dittert\inst{1}}
\authorrunning{M.~Beckmann, C.~Dittert}
\institute{Center for Industrial Mathematics, University of Bremen, \\
Bibliothekstra{\ss}e 5, 28359 Bremen, Germany\\
\email{\{research@mbeckmann.de,cadi@uni-bremen.de\}}
\and
Department of Electrical and Electronic Engineering, Imperial College London, \\
London, SW72AZ, UK}
\maketitle
\begin{abstract}
This paper proves a novel analytical inversion formula for the so-called modulo Radon transform (MRT), which models a recently proposed approach to one-shot high dynamic range tomography.
It is based on the solution of a Poisson problem linking the Laplacian of the Radon transform (RT) of a function to its MRT in combination with the classical filtered back projection formula for inverting the RT.
Discretizing the inversion formula using Fourier techniques leads to our novel Laplacian Modulo Unfolding - Filtered Back Projection algorithm, in short LMU-FBP, to recover a function from fully discrete MRT data.
Our theoretical findings are finally supported by numerical experiments.

\end{abstract}

\section{Introduction}
Computerized tomography (CT) deals with the recovery of the interior of scanned objects from given X-ray measurements taken at different views.
By design, the detector's dynamic range is limited leading to saturation artifacts in the reconstruction of high-contrast objects.
The goal of increasing the dynamic range points towards the recent topic of high dynamic range (HDR) tomography.
Inspired by consumer HDR photography, typical multi-exposure approaches combine several low dynamic range measurements into a single image with increased dynamic range~\cite{Chen2015}.
To avoid the drawback of manual calibration of the detector for each exposure, an automated calibration approach is proposed in~\cite{Chen2020,Li2023}. 
Unfortunately, multi-exposure approaches suffer from the same difficulties as HDR photography: ghosting artifacts due to movements of the investigated object; increased acquisition time; unknown detector's sensor response necessary for satisfactory image fusion.
Furthermore, the mentioned approaches are solely based on empirical experiments lacking mathematical recovery guarantees.

As opposed to this, inspired by the Unlimited Sampling (US) framework~\cite{Bhandari2020,Bhandari2021}, a recently introduced single-exposure approach based on a co-design of hardware and algorithms allows for mathematically backed recovery strategies for HDR CT \cite{Beckmann2022,Beckmann2024}.
Here, instead of truncating measurements that exceed the dynamic range, these are folded into the limited range of the detector via a modulo operation implemented in hardware.
Thereon, these folded measurements are unfolded algorithmically.
Mathematically, this can be modeled by the so-called modulo Radon transform (MRT)~\cite{Beckmann2022}, which concatenates the modulo operation with the well-known Radon transform (RT), yielding a nonlinear reconstruction problem.
Albeit algorithmic recovery strategies exist~\cite{Beckmann2022,Beckmann2024}, an explicit inversion formula for the MRT is still missing.
In this work, we address this point by proposing an analytical inversion formula for the MRT, that combines the solution of a Poisson equation to unfold the modulo operation with the well-known filtered back projection (FBP) formula to invert the RT.
The formulation of the Poisson equation is inspired by an approach to the classical phase unwrapping problem in~\cite{Robinson2016,Schofield2003}.
By discretizing our analytical inversion formula, we deduce the novel LMU-FBP algorithm for recovering a target from MRT data, which operates simultaneously on the angle and radial variable in contrast to the US-FBP algorithm~\cite{Beckmann2022}.
Moreover, our inversion formula does not require bandlimited Radon data so that LMU-FBP does not rely on bandlimitedness as opposed to~\cite{Beckmann2024}.

The paper starts with an overview of the MRT in §\ref{sec:MRT}.
This is followed by our main theoretical result: the deduction of the Poisson equation with suitable boundary conditions in §\ref{sec:Poisson_equation} and the proof of the analytical inversion in §\ref{sec:analytical_inversion}.
Finally, we present our LMU-FBP algorithm in §\ref{sec:numerical_inversion} and showcase numerical experiments with non-bandlimited test cases in §\ref{sec:numerical_experiments} including a comparison with US-FBP~\cite{Beckmann2022}.

\section{Modulo Radon Transform}\label{sec:MRT}

For $f \in \Lebesgue^1(\R^2)$, we define its {\em Radon transform} (RT) $\Radon f \in \Lebesgue^1((-\pi,\pi) \times \R)$ via
\begin{equation*}
\Radon f (\vartheta, t) \coloneq \int_{\R} f(t \cos(\vartheta) - s \sin(\vartheta), t \sin(\vartheta) + s \cos(\vartheta)) \: \d s.
\end{equation*}
Moreover, for $g \in \Lebesgue^\infty((-\pi, \pi) \times \R)$, we define its {\em back projection} $\Back g \in \Lebesgue^\infty(\R^2)$ via
\begin{equation*}
\Back g(x) \coloneqq  \int_{- \pi}^{\pi} g(\vartheta, x_1 \cos(\vartheta) + x_2\sin(\vartheta)) \: \d \vartheta.
\end{equation*}
Note that $\Radon$ is injective on $\Lebesgue^1(\R^2)$ and, for $f \in \Lebesgue^1(\R^2) \cap \Cont(\R^2)$ with integrable Fourier transform $\Fourier f \in \Lebesgue^1(\R^2)$, the inversion of $\Radon$ is given by the classical {\em filtered back projection formula}
\begin{equation}\label{eq:fbp_formula}
f(x) = \frac{1}{4 \pi} \Back \left( \Fourier^{-1} \left[\abs{S} \Fourier\left(\Radon f\right)(\vartheta, S)\right]\right)(x),
\end{equation}
which holds pointwise for all $x\in \R^2$.
The FBP formula~\eqref{eq:fbp_formula}, however, requires exact knowledge of $\Radon f$ and, in particular, saturation effects due to range limitations lead to severe artefacts in the reconstruction.
To circumvent this, in~\cite{Beckmann2022} the following {\em modulo Radon transform} is introduced, which folds $\Radon f$ into a given range interval $[-\lambda,\lambda]$ with $\lambda > 0$.

\begin{definition}[Modulo Radon transform]
For $\lambda > 0$ and $f \in \Lebesgue^1(\R^2)$ the {\em modulo Radon transform} (MRT) $\Radon^\lambda f : (-\pi, \pi) \times \R \rightarrow [-\lambda, \lambda]$ is defined as
\begin{equation*}
\Radon^\lambda f (\vartheta, t) \coloneqq \Mod^\lambda(\Radon f (\vartheta, t)),
\end{equation*}
where $\Mod^\lambda: \R \to [-\lambda,\lambda]$ is the {\em $2\lambda$-modulo operator} with $\Mod^\lambda (t) \coloneqq t - 2\lambda \floor{ \frac{t + \lambda}{2\lambda}}$.
\end{definition}

In~\cite{Beckmann2022}, it is shown that $\Radon^\lambda$ is injective on the space $\Bernstein^1_L(\R^2)$ of bandlimited integrable functions with bandwidth $L$ for any $L>0$ and on the space $\Cont_c(\R^2)$ of continuous functions with compact support, which implies the invertibility of $\Radon^\lambda$ on suitable spaces.
An analytical inversion formula, however, is not known so far.
In the following sections, we take steps towards closing this gap by first relating the Laplacian of $\Radon f$ to $\Radon^\lambda f$ and, afterwards, showing  unique solvability of the corresponding Poisson problem.
This in combination with the FBP formula~\eqref{eq:fbp_formula} then gives a new analytical inversion formula for $\Radon^\lambda$ under suitable assumptions.

\section{Poisson Problem for the Modulo Radon Transform}\label{sec:Poisson_equation}

Relating the Laplacian of $\Radon f$ to $\Radon^\lambda f$ via a Poisson problem is inspired by~\cite{Schofield2003} from the context of phase unwrapping.
Let $f\in \Lebesgue^1(\R^2)$ such that $\Radon f$ is twice continuously differentiable on $(-\pi, \pi)\times \R$.
Furthermore, let $\lambda > 0$ be the modulo threshold.
Then, the Laplacian of $\exp{\i \frac{\pi}{\lambda}\Radon f}$ can be computed as
\begin{equation*}
\laplace \exp{\i\frac{\pi}{\lambda} \Radon f} = \Biggl[\i \frac{\pi}{\lambda} \laplace \Radon f - \biggl(\frac{\pi}{\lambda} \deriv[]{\vartheta} \Radon f\Biggr)^2 - \biggl(\frac{\pi}{\lambda} \deriv[]{t} \Radon f\biggr)^2\biggr] \exp{\i\frac{\pi}{\lambda} \Radon f},
\end{equation*}
and, thereon, for the Laplacian of $\Radon f$ follows that
\begin{equation}\label{eq:laplace_radon}
\laplace \Radon f = \frac{\lambda}{\pi} \left[\cos\Bigl(\frac{\pi}{\lambda} \Radon f\Bigr) \laplace \sin\Bigl(\frac{\pi}{\lambda} \Radon f\Bigr) - \sin\Bigl(\frac{\pi}{\lambda} \Radon f\Bigr) \laplace \cos\Bigl(\frac{\pi}{\lambda} \Radon f\Bigr)\right].
\end{equation}
Observe that $\Radon f$ can be decomposed pointwise into $\Radon^\lambda f$ and a piecewise constant residual function with values in $2\lambda\Z$, i.e.,
\begin{equation} \label{eq:mod_decomposition}
\Radon f = \Radon^{\lambda} f +  2 \lambda \varepsilon_{\Radon f}
\quad \text{with} \quad
\varepsilon_{\Radon f} = \sum\nolimits_{i \in \mathcal{I}} c_i \mathrm{1}_{D_i}.
\end{equation}
Here, $\mathcal{I}$ denotes an arbitrary index set such that $\R^2$ is covered by the pairwise disjunct sets $\{D_i\}_{i \in \mathcal{I}}$, and the coefficients $c_i \in \Z$ are integers.
Exploiting the $2\pi$-periodicity of the sine and cosine function in~\eqref{eq:laplace_radon} results in the {\em Poisson equation}
\begin{equation}\label{eq:poisson_equation}
\laplace \Radon f = \frac{\lambda}{\pi} \left[ \cos\Bigl(\frac{\pi}{\lambda} \Radon^{\lambda} f\Bigr) \laplace \sin\Bigl(\frac{\pi}{\lambda} \Radon^{\lambda} f\Bigr) - \sin\Bigl(\frac{\pi}{\lambda} \Radon^{\lambda} f\Bigr) \laplace \cos\Bigl(\frac{\pi}{\lambda} \Radon^{\lambda} f\Bigr) \right].
\end{equation}
To compute $\Radon f$ from given $\Radon^\lambda f$, we aim to solve the Poisson equation~\eqref{eq:poisson_equation}.
For this, we restrict the domain of $\Radon f$ to a bounded rectangle and impose boundary conditions.
Note that, for applications like CT, it is reasonable to assume that the function $f$ is compactly supported.
Hence, in the following, let $f \in \Lebesgue^1(\R^2)$ have compact support in the open unit ball $\Ball_1(0)\subseteq\R^2$.
As the compact support of $f$ transfers to $\Radon f$, we obtain $\supp \Radon f(\vartheta, \cdot) \subseteq (-1,1)$ for all $\vartheta \in (-\pi, \pi)$.
Consequently, we assume $f \in \Lebesgue^1(\Ball_1(0))$ as in~\cite{Rieder2000},~\cite{Natterer2001} and extend $f$ by zero if needed.
On the restricted Lipschitz domain $\Omega = (-\pi, \pi) \times (-1,1)$, the RT $\Radon f: \Omega \rightarrow \R$ can be represented as
\begin{equation*}
\Radon f(\vartheta, t) = \int_{\ell_{\theta(\vartheta), t} \cap \Ball_1(0)} f(x) \d x
\end{equation*}
with $\theta(\vartheta) = (\cos(\vartheta), \sin(\vartheta))^\perp$.
With this, $\Radon f$ satisfies homogeneous Dirichlet boundary conditions on $
\Gamma_D \coloneqq (-\pi, \pi) \times \{-1\} \cup (-\pi, \pi) \times \{1\} \subseteq \partial \Omega$,
where $\partial \Omega$ denotes the boundary of $\Omega$.
Furthermore, $\Radon f$ is $2 \pi$-periodic with respect to $\vartheta$.
Let $\Radon f$ also be continuously differentiable with respect to the first argument on $[-\pi, \pi] \times \R$.
Then, $\Radon f$ and its first partial derivative with respect to $\vartheta$ have to satisfy $2\pi$-periodic boundary conditions on $\Gamma_P \coloneqq \{-\pi\} \times (-1, 1) \cup \{\pi\} \times (-1,1) \subseteq \partial \Omega$.
Altogether, $\Radon f$ has to fulfill the {\em mixed Dirichlet-periodic boundary conditions}
\begin{equation}\label{eq:boundary_conditions}
\begin{cases}
\Radon f (\vartheta, -1)               = \Radon f (\vartheta, 1) = 0         & \text{ for } \vartheta \in (-\pi, \pi) \\
\Radon f(- \pi, t)                     = \Radon f(\pi, t)                    & \text{ for } t \in (-1,1)             \\
\deriv[]{\vartheta}\Radon f (-\pi, t)  = \deriv[]{\vartheta}\Radon f(\pi, t) & \text{ for }t\in (-1,1).
\end{cases}
\end{equation}
In total, by abbreviating the right-hand side in~\eqref{eq:poisson_equation} as
\begin{equation*}
\Radon^\lambda_\laplace(f) \coloneqq \left(\tfrac{\lambda}{\pi}\left[\cos(\tfrac{\pi}{\lambda}\Radon^\lambda f)\laplace \sin(\tfrac{\pi}{\lambda}\Radon^\lambda f) - \sin(\tfrac{\pi}{\lambda}\Radon^\lambda f)\laplace\cos(\tfrac{\pi}{\lambda}\Radon^\lambda f)\right]\right)
\end{equation*}
and setting $\Omega_\vartheta = (-\pi,\pi)$, $\Omega_t = (-1,1)$, we obtain the {\em Poisson problem}
\begin{equation}\label{eq:poisson_problem}
\left\{\begin{aligned}
\laplace \Radon f & =\Radon^\lambda_\laplace(f) && \text{on } \Omega \\
\Radon f (\cdot, -1) & = 0, \; \Radon f (\cdot, 1) = 0 && \text{on } \Omega_\vartheta \\
\Radon f(- \pi, \cdot) & = \Radon f(\pi, \cdot); \; \tfrac{\partial}{\partial\vartheta}\Radon f (-\pi, \cdot) = \tfrac{\partial}{\partial\vartheta}\Radon f(\pi, \cdot) && \text{on } \Omega_t.
\end{aligned}\right.
\end{equation}

\section{Analytical Inversion of the Modulo Radon Transform}\label{sec:analytical_inversion}

The solution of the boundary value problem~\eqref{eq:poisson_problem} enables the inversion of the modulo operation.
To guarantee the existence of a unique solution, we employ certain periodic function spaces.
A periodic function is defined on the torus $\T^2 = [-\pi, \pi]\times [-1,1] $ such that opposite points are identified with each other, analogously to~\cite[§~9.1.1]{Triebel1983}.
With this, the {\em Lebesgue space $\Lper\left(\T^2\right)$ of periodic and square-integrable functions} is defined as
\begin{equation*}
\Lper\left(\T^2\right) \coloneqq \set{f: \T^2\rightarrow \R \text{ periodic, measurable}}{\norm{f}{\Lper} < \infty}
\end{equation*}
with norm
\begin{equation*}
\norm{f}{\Lper} \coloneqq \left(\int_{\T^2} \abs{f(x)}^2 \d x\right)^{\frac{1}{2}}
\end{equation*}
while identifying functions which agree almost everywhere on $\T^2$.
Let $\partial_{\vartheta}$ and $\partial_t$ denote the first weak partial derivative with respect to $\vartheta$ and $t$, respectively.
Then, the {\em periodic Sobolev space $\HperT{2}$} on $\T^2$ is defined as
\begin{equation*}
\HperT{2} \coloneqq \set{f \in \Lper\left(\T^2\right)}{ \partial_{\vartheta}f, ~  \partial_{t} f \in \Lper\left(\T^2\right) }
\end{equation*}
with
\begin{equation*}
\norm{f}{\Hper} \coloneqq \left(\norm{f}{\Lper}^2 + \norm{\partial_\vartheta f}{\Lper}^2 + \norm{\partial_t f}{\Lper}^2\right)^{\frac{1}{2}}.
\end{equation*}
Furthermore, we define the subspace $\HperDT{2}$ that encodes homogeneous Dirichlet boundary conditions on a part of the boundary as the closure
\begin{equation*}
\HperDT{2} \coloneqq \closure{\cperD{\infty}{2}}^{\Hper}
\end{equation*}
of the space $\cperD{\infty}{2}$ of periodic smooth functions vanishing in a neighborhood of $\Gamma_D$
\begin{equation*}
\cperD{\infty}{2} \coloneqq \set{\varphi \in \Cont^{\infty}_{\per}(\T^2)}{ \begin{matrix} \exists \, 0 < r < 1 \; \forall \, \vartheta \in (-\pi,\pi)\colon \\
\supp \varphi(\vartheta, \cdot) \subseteq (-r, r)\end{matrix}}.
\end{equation*}
Employing standard arguments, the well-known Poincaré inequality can be adapted to the setting of $\HperDT{2}$.
\begin{proposition}[Poincaré inequality]\label{prop:poincare}
There exists a constant $C_p>0$ such that
\begin{equation*}
\norm{u}{\Lper} \leq C_p \norm{\nabla u}{\Lper}
\quad \forall \, u\in \HperDT{2}.
\end{equation*}
\end{proposition}
Due to space limitations, we omit the proof.
Instead, we now formulate our main theoretical result: an analytical inversion formula for the MRT.

\begin{theorem}[Inversion formula for $\Radon^\lambda$]\label{theo:inversion_mrt}
Let $f \in \Lebesgue^1(\Ball_1(0)) \cap \Cont(\closure{\Ball_1(0)})$ be compactly supported in $\Ball_1(0)$ with $\Fourier f \in \Lebesgue^1(\R^2)$.
Moreover, let $\Radon f \in \Cont^2(\overline{\Omega})$ with $\Omega = (-\pi, \pi)\times (-1,1)$ and set
\begin{equation*}
\Radon^\lambda_\laplace(f) \coloneqq \left(\tfrac{\lambda}{\pi}\left[\cos(\tfrac{\pi}{\lambda}\Radon^\lambda f)\laplace \sin(\tfrac{\pi}{\lambda}\Radon^\lambda f) - \sin(\tfrac{\pi}{\lambda}\Radon^\lambda f)\laplace\cos(\tfrac{\pi}{\lambda}\Radon^\lambda f)\right]\right).
\end{equation*}
Then,
\begin{equation} \label{eq:mrt_inverse}
f(x)  = \frac{1}{4 \pi} \Back \left( \Fourier_1^{-1} \left[\abs{S} \Fourier_1\left(\laplace^{-1}\Radon^\lambda_\laplace(f)\right)(\vartheta, S)\right]\right)(x)
\end{equation}
holds for all $x \in \R^2$, where $\laplace^{-1} \Radon^\lambda_\laplace(f)$ is the weak solution to~\eqref{eq:poisson_problem} for given $\Radon^\lambda f$.
\end{theorem}

For the proof of the Theorem~\ref{theo:inversion_mrt}, we first show that the Poisson problem~\eqref{eq:poisson_problem} admits a unique weak solution.
To this end, we deduce a weak formulation of the boundary value problem, where our calculations are inspired by~\cite[§5.2,~6.5,~7.4]{Arendt2023} and~\cite[§6.8.2,~8.3]{Salsa2022}.
Assume that $f\in \Lebesgue^1(\Ball_1(0))$, $\Radon^\lambda_\laplace(f) \in \Cont(\closure{\Omega})$ and $\Radon f \in \Cont^2(\closure{\Omega})$. 
Then, $\Radon f$ is a classical solution to the boundary value problem in~\eqref{eq:poisson_problem}.
Analogously to~\cite{Rieder2000}, we periodize $\Radon f$ with respect to $t$, since the values on the opposing boundary parts $(-\pi, \pi) \times \{-1\}$ and $(-\pi, \pi) \times \{1\}$ agree due to the Dirichlet conditions in~\eqref{eq:boundary_conditions}.  
Consequently, multiplying the Poisson equation~\eqref{eq:poisson_equation} with a test function $ \varphi \in \cperD{\infty}{2}$ and applying Green's first identity results in
\begin{equation*}
-\int_{\T^2} \varphi \, \Radon^\lambda_\laplace(f) \: \d (\vartheta, t)  = \int_{\Omega} \nabla \varphi \cdot \nabla \Radon f \: \d (\vartheta, t) - \int_{\partial \Omega} \varphi \, \nabla \Radon f \cdot \nu \: \d \sigma,
\end{equation*}
where the integral along the boundary vanishes. 
To see this, we use that the test function $\varphi$ is periodic and vanishes in a neighborhood of $\Gamma_D$, as well as the condition that the partial derivative $\partial_{\vartheta}\Radon f$ is $2 \pi$-periodic with respect to $\vartheta$.
Finally, we enlarge the set of test functions to $\HperDT{2}$ to obtain the desired weak formulation:
For $\Radon^\lambda_\laplace(f) \in \Lper(\T^2)$, a function $u \in \HperDT{2}$ is called a {\em weak solution} to the Poisson problem~\eqref{eq:poisson_problem} if
\begin{equation}\label{eq:weak_formulation}
\int_{\T^2} \nabla u(\vartheta,t) \cdot \nabla v(\vartheta,t) \: \d (\vartheta, t) = - \int_{\T^2} v(\vartheta,t) \, \Radon^\lambda_\laplace(f)(\vartheta,t) \: \d (\vartheta, t)
\end{equation}
holds for all $v \in \HperDT{2}$.

\begin{lemma}[Weak solution] \label{lem:weak_solution}
Let $f \in \Lebesgue^1(\Ball_1(0))$ with $\Radon^\lambda_\laplace(f) \in \Lper(\T^2)$.
Then, the Poisson problem~\eqref{eq:poisson_problem} has a unique weak solution $u_0 \in \HperDT{2}$ satisfying~\eqref{eq:weak_formulation}.
Moreover, there exists a constant $C>0$ such that the stability estimate
\begin{equation}\label{eq:stability}
\norm{u_0}{\Hper} \leq C \norm{\Radon^\lambda_\laplace(f)}{\Lper}
\end{equation}
holds.
\end{lemma}

\begin{proof}
We follow a standard approach based on the classical Lax-Milgram theorem, see e.g.~\cite[Chapter~5.2,~6.5~7]{Arendt2023},~\cite[Chapter~8.3]{Salsa2022}.
To this end, we first define the bilinear form $\alpha: \HperDT{2} \times \HperDT{2} \rightarrow \R$ via
\begin{equation*}
\alpha(u,v) \coloneqq \int_{\T^2} \nabla u(\vartheta,t) \cdot \nabla v(\vartheta,t) \: \d(\vartheta,t).
\end{equation*}
Due to the Poincaré inequality, Proposition~\ref{prop:poincare}, the mapping $\alpha$ is an inner product on $\HperDT{2}$ and the Cauchy-Schwarz inequality implies the continuity of $\alpha$ by
\begin{equation*}
\abs{a(u,v)} \leq \norm{\nabla u}{\Lper} \, \norm{\nabla v}{\Lper} \leq \norm{u}{\Hper} \, \norm{v}{\Hper}.
\end{equation*}
Furthermore, the Poincaré inequality guarantees the existence of $C_p>0$ such that
\begin{equation*}
\alpha(u,u) \geq \frac{1}{2} \norm{\nabla u }{\Lper}^2 + \frac{1}{2C_p^2}  \norm{u}{\Lper}^2 \geq \min\left\{\frac{1}{2}, \frac{1}{2 C_p^2}\right\} \norm{u}{\Hper}^2.
\end{equation*}
This shows the coercivity of $\alpha$ on $\HperDT{2}$.
Secondly, we define the linear form $F: \HperDT{2} \rightarrow \R$ by
\begin{equation*}
F(v) \coloneqq - \int_{\T^2} v(\vartheta,t) \, \Radon^\lambda_\laplace(f)(\vartheta,t) \: \d(\vartheta,t).
\end{equation*}
Again, the Cauchy-Schwarz inequality implies the continuity of $F$ by
\begin{equation*}
\abs{F(v)} \leq \norm{\Radon^\lambda_\laplace(f)}{\Lper} \, \norm{v}{\Lper} \leq \norm{\Radon^\lambda_\laplace(f)}{\Lper} \, \norm{v}{\Hper}.
\end{equation*}
Consequently, all requirements of the Lax-Milgram theorem are satisfied, and there exists a unique $u_0 \in \HperDT{2}$ such that
\begin{equation*}
\int_{\T^2} \nabla u_0 \cdot \nabla v \: \d(\vartheta,t) = \alpha (u_0, v) = F(v) = -\int_{\T ^2} v \, \Radon^\lambda_\laplace(f) \: \d(\vartheta,t)
\end{equation*}
is satisfied for all $v \in \HperDT{2}$, i.e., $u_0$ is the unique weak solution to the Poisson problem~\eqref{eq:poisson_problem}.
Furthermore, the coercivity of $\alpha$ and the continuity of $F$ imply that
\begin{equation*}
\min\left\{\frac{1}{2}, \frac{1}{2 C_p^2}\right\} \norm{u_0}{\Hper}^2 \leq \abs{\alpha(u_0, u_0)} = \abs{F(u_0)} \leq \norm{\Radon^\lambda_\laplace(f)}{\Lper} \norm{u_0}{\Hper}
\end{equation*}
and, thus, dividing by $\norm{u_0}{\Hper}$ gives the stability estimate~\eqref{eq:stability}.\qed
\end{proof}

Using the unique weak solution in Lemma~\ref{lem:weak_solution}, we can now prove our main theorem.

\begin{proof}[Theorem~\ref{theo:inversion_mrt}]
In the first step, we invert the modulo operator by showing that the equation 
\begin{equation}\label{eq:mod_inverse}
\Radon f = \laplace^{-1} \left(\Radon^\lambda_\laplace(f)\right)
\end{equation}
holds pointwise on $\Omega$.
Since $\Radon f \in \Cont^2(\overline{\Omega})$ by assumption, the function $\Radon^\lambda_\laplace(f)$ is square-integrable.
Moreover, since $\Radon^\lambda f$ is periodic with respect to $\vartheta$ and, after periodization, also periodic with respect to $t$, $\Radon^\lambda_\laplace(f)$ is periodic on $\Omega$ and, thus, $\Radon^\lambda_\laplace(f) \in \Lper\left(\T^2\right)$.
Hence, Lemma~\ref{lem:weak_solution} implies the existence of a unique weak solution  $u_0 \coloneqq \laplace^{-1}(\Radon^\lambda_\laplace(f)) \in \HperDT{2}$ to the Poisson problem~\eqref{eq:poisson_problem}.
It remains to argue that $\Radon f$ and $u_0$ coincide.
By assumption, we have $\Radon f \in \cperD{2}{2}$ and, hence, $\Radon f \in \HperDT{2}$.
Furthermore, $\Radon f$ satisfies the Poisson equation~\eqref{eq:poisson_equation} and the boundary conditions~\eqref{eq:boundary_conditions}.
Following the deduction of~\eqref{eq:weak_formulation}, $\Radon f$ is also a weak solution of~\eqref{eq:poisson_problem} and, due to the uniqueness of the weak solution, it follows that $ u_0 = \Radon f \in \HperDT{2}$.
Consequently, the continuity of $\Radon f$ implies that~\eqref{eq:mod_inverse} holds pointwise choosing the continuous representative.

In the second step, we need to invert the Radon operator.
Since $f\in \Cont(\closure{\Ball_1(0)})$ and $\Fourier f \in \Lebesgue^1(\R^2)$ by assumption, the filtered back projection formula~\eqref{eq:fbp_formula} holds pointwise on $\R^2$ and, consequently, the explicit inversion formula~\eqref{eq:mrt_inverse} for the MRT follows by combining~\eqref{eq:mod_inverse} with~\eqref{eq:fbp_formula}. \qed
\end{proof}

\section{Numerical Inversion of the Modulo Radon Transform}\label{sec:numerical_inversion}

We approximate the analytical inversion formula~\eqref{eq:mrt_inverse} numerically by solving the Poisson problem~\eqref{eq:poisson_problem} with Fourier techniques to invert the modulo operation and applying the well-known discrete FBP algorithm to invert the RT.
This combines into our novel {\em Laplacian Modulo Unfolding - Filtered Back Projection} (LMU-FBP) Algorithm~\ref{alg:lmu_fbp}, which we now explain in more detail.

The MRT is discretized using parallel beam geometry~\cite{Natterer2001}, where, due to the evenness of $\Radon f$, it suffices to consider the domain $[0,\pi) \times [-1,1]$.
For this, let $K,M \in \N$ and set $N = 2K+1$. 
Moreover, let $T>0$ be the radial sampling rate.
Then, $\Radon^\lambda f$ is discretized by evaluating at the grid points $\vartheta_m = m\frac{\pi}{M}$, $m\in \idxset{0}{M-1}$, and $t_n = (n-K) T$, $n \in \idxset{0}{N-1}$, resulting in the discrete MRT data
\begin{equation*}
\set{p^\lambda[m, n]\coloneqq \Radon^\lambda f (\vartheta_m, t_n)}{m \in \idxset{0}{M-1}, ~ n \in \idxset{0}{N-1}}.
\end{equation*}

Inspired by~\cite{Schofield2003}, in the first {\em Laplacian Modulo Unfolding} (LMU) stage of Algorithm~\ref{alg:lmu_fbp} we invert the modulo operator numerically by solving the Poisson equation~\eqref{eq:poisson_equation} using discrete Fourier transforms.
This is based on the observation that, under suitable assumptions, the Laplacian of a function $p$ can be computed via
\begin{equation*}
\laplace p = \Fourier^{-1} \left(-\abs{\cdot}^2 \Fourier p \right).
\end{equation*}
To incorporate the boundary conditions~\eqref{eq:boundary_conditions}, the MRT data is extended:
In Step~\ref{alg:periodic_conditions}, the MRT data is extended in $\vartheta$ to the interval $\left[0,2\pi\right)$ such that it becomes $2\pi$-periodic.
For this, the evenness property of $\Radon$ is employed.
To ensure the homogeneous Dirichlet boundary conditions, in Step~\ref{alg:Dirichlet_conditions}, the MRT data is extended in $t$ such that it becomes odd around $N+1$.
Using the discrete Fourier transform ($\DFT$) and its inverse ($\IDFT$), the discrete version of the right-hand side of the Poisson equation~\eqref{eq:poisson_equation} is computed in Step~\ref{alg:rhs}.
Finally, in Step~\ref{alg:PDE}, the Poisson equation is solved numerically and the LMU solution $p_{\LMU}$ is found in Step~\ref{alg:restriction} by restricting to the original index set $\idxset{0}{M-1} \times \idxset{0}{N-1}$.

In the second {\em Filtered Back Projection} (FBP) stage of Algorithm~\ref{alg:lmu_fbp} we apply the discrete FBP algorithm to numerically invert the RT.
This is based on the approximate FBP reconstruction formula
\begin{equation} \label{eq:approx_fbp}
f_{\FBP} = \frac{1}{4\pi} \Back(A_L * \Radon f),
\end{equation}
where a low-pass filter $A_L$ satisfying $\Fourier A_L(S) = |S| \, W(\frac{S}{L})$ with an even window $W \in \Lebesgue^\infty(\R)$ supported in $[-1,1]$ and bandwidth $L>0$ is incorporated to deal with the ill-posedness of the Radon inversion.
The approximate FBP forumla~\eqref{eq:approx_fbp} is discretized using a standard approach, cf.~\cite{Natterer2001}, involving the discrete convolution in Step~\ref{alg:convolution} followed by the discrete back projection in Step~\ref{alg:back_projection}, where an interpolation method $\Int$ is applied to reduce the computational costs.
To this end, the discrete convolution is computed at $t_i$, $i \in \mathcal{I}$, for a sufficiently large index set $\mathcal{I} \subset \Z$.
The result is the LMU-FBP reconstruction $f_{\LMU}$ in grid points $(x_i,y_j) \in \R^2$ for index sets $\mathcal{I}_x, \mathcal{I}_y \subset \N$.

\begin{algorithm}[t]
\caption{LMU-FBP}
\label{alg:lmu_fbp}
\textbf{Input:}
MRT data $p^\lambda[m,n]$ for $m \in \{0, \dots, M-1\}$, $n \in \{0, \dots, N-1\}$;\\
low-pass filter $A_L$ with bandwidth $L > 0$; interpolation method $\Int$

\smallskip
{\small \textbf{Laplacian Modulo Unfolding (LMU):}}
\begin{algorithmic}[1]
\scriptsize
\STATE {\label{alg:periodic_conditions}
\textbf{for} $m \in \idxset{0}{2M-1}$, $n \in \idxset{0}{N-1}$ \textbf{do}\\[0.5ex]
\qquad$\tilde{p}^\lambda[m, n] \gets \begin{cases}
p^\lambda[m, n] & \text{if } m \in \{0, \dots, M-1\} \\
p^\lambda[m-M, N-1-n] & \text{if } m \in\{M, \dots, 2M-1\}
\end{cases}$\\
}
\STATE {\label{alg:Dirichlet_conditions}
\textbf{for} $m\in\idxset{0}{2M-1}$, $n \in \idxset{0}{2N+1}$ \textbf{do}\\[0.5ex]
\qquad$\hat{p}^\lambda[m, n] \gets \begin{cases}
0 & \text{if } n \in \{0, N+1\}              \\
\tilde{p}^\lambda[m, n-1] & \text{if } n \in \idxset{1}{N}      \\
-\tilde{p}^\lambda[m, 2N+1-n] & \text{if } n \in \idxset{N+2}{2N+1} \\
\end{cases}$\\
}
\STATE{\label{alg:rhs}
$\hat{p}^\lambda_{\mathrm{s}} \gets \IDFT[-(2\pi\abs{\cdot})^2 \DFT[\sin(\frac{\pi}{\lambda}\hat{p}^\lambda)]]$, \;
$\hat{p}^\lambda_{\mathrm{c}} \gets \IDFT[-(2\pi\abs{\cdot})^2 \DFT[\cos(\frac{\pi}{\lambda}\hat{p}^\lambda)]]$, \\
$\hat{p}^\lambda_\laplace \gets \frac{\lambda}{\pi}[\cos(\frac{\pi}{\lambda}\hat{p}^\lambda) \hat{p}^\lambda_{\mathrm{s}} - \sin(\frac{\pi}{\lambda}\hat{p}^\lambda) \hat{p}^\lambda_{\mathrm{c}}]$}\\[0.5ex]
\STATE{\label{alg:PDE}
$\hat{p} \gets \IDFT[- \left(2\pi \abs{\cdot}\right)^{-2} \DFT [\hat{p}^\lambda_\laplace]]$}\\[0.5ex]
\STATE{\label{alg:restriction}
\textbf{for} $m\in\idxset{0}{M-1}$, $ n \in \idxset{0}{N-1}$ \textbf{do}\\[0.5ex]
\qquad$p_{\LMU}[m,n] \gets \hat{p}[m, n+1]$\\
}
\end{algorithmic}

{\small \textbf{Filtered Back Projection (FBP):}}
\begin{algorithmic}[1]
\scriptsize
\setcounter{ALC@line}{5}
\STATE{\label{alg:convolution}
\textbf{for} $m\in\idxset{0}{M-1}$, $i \in \mathcal{I}$ \textbf{do}\\[0.5ex]
\qquad$h(\vartheta_m, t_i) \gets T \sum\limits_{n=0}^{N-1} \Fourier^{-1} A_\bandw(t_i - t_n) \, p_{\LMU} [m,n]$
}
\STATE{\label{alg:back_projection}
\textbf{for} $(i,j) \in \mathcal{I}_x \times \mathcal{I}_y$ \textbf{do}\\[0.5ex]
\qquad$f_{\LMU}[i,j] \gets \dfrac{1}{2M} \sum\limits_{m=0}^{M-1} \Int h(\vartheta _m, x_i \cos(\vartheta_m) + y_j \sin(\vartheta_m))$
}
\end{algorithmic}
\textbf{Output:}
LMU-FBP reconstruction $f_{\LMU}(x_i, y_j)$ for $(i, j) \in \mathcal{I}_x \times \mathcal{I}_y$
\end{algorithm}

\smallskip
\noindent
{\bf Improvement Step.}
In~\cite{Schofield2003}, it is proposed to apply an enhancement rounding step, which we adapt to our setting and include after the LMU stage in Algorithm~\ref{alg:lmu_fbp}.
More precisely, our {\em improvement step} is defined as
\begin{equation}\label{eq:improvement_step}
p_{\LMU_+} [m, n] = p^\lambda[m,n] + 2 \lambda \round\Bigl(\frac{p_{\LMU}[m,n] - p^\lambda[m,n]}{2 \lambda}\Bigr)
\end{equation}
for $m \in \idxset{0}{M-1}$ and $n\in \idxset{0}{N-1}$.
This yields exact recovery of the Radon data $p[m,n] = \Radon f(\vartheta_m,t_n)$ if the absolute LMU reconstruction error satisfies $\abs{p_{\LMU}[m, n] - p[m,n]} < \lambda$.
Indeed, using the modulo decomposition property~\eqref{eq:mod_decomposition} with piecewise constant residual $\varepsilon_p[m,n] = \varepsilon_{\Radon f}(\vartheta_m,t_n) \in \Z$, it is
\begin{align*}
p_{\LMU_+}[m,n] &= p^\lambda[m,n] + 2 \lambda \round\Bigl(\frac{p_{\LMU}[m,n] - p[m,n] + 2 \lambda \varepsilon_p[m,n]}{2\lambda}\Bigr) \\
& = p^\lambda[m,n] + 2 \lambda \varepsilon_p[m,n] = p[m,n].
\end{align*}
However, if the absolute LMU reconstruction error is large, then the improvement step in~\eqref{eq:improvement_step} yields undesirable jumps in the recovered Radon data.

\section{Numerical Experiments}\label{sec:numerical_experiments}

\begin{figure}[t]
\centering
\includegraphics[width=\textwidth]{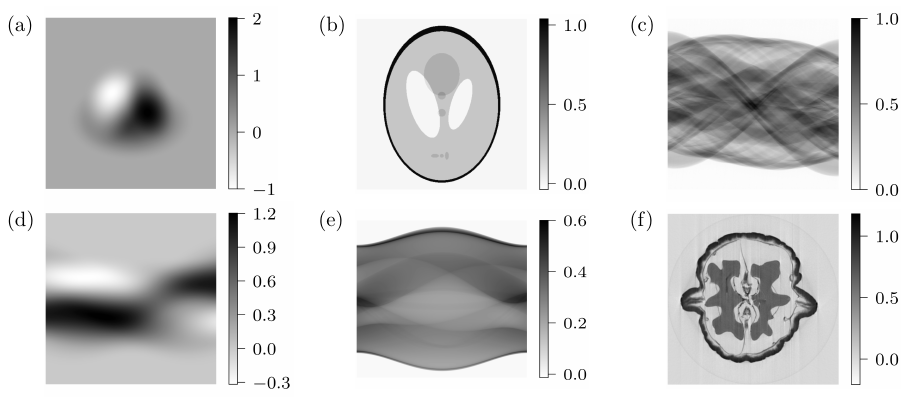}
\caption{Utilized test data for numerical experiments. (a) Smooth phantom from~\cite{Rieder2003}. (b) Shepp-Logan phantom from~\cite{Shepp1974}. (c) Walnut Radon data from~\cite{Siltanen2015}. (d) Radon data of (a). (e) Radon data of (b). FBP reconstruction of (c) serving as ground truth.}
\label{fig:phantoms}
\end{figure}

We now present numerical experiments to demonstrate our inversion approach.
To this end, we use the smooth phantom~\cite{Rieder2003}, depicted in Fig.~\ref{fig:phantoms}(a) along with its Radon data in Fig.~\ref{fig:phantoms}(d), and the classical Shepp-Logan phantom~\cite{Shepp1974} in Fig.~\ref{fig:phantoms}(b), whose Radon data is shown in Fig.~\ref{fig:phantoms}(e).
We also consider the open source walnut dataset~\cite{Siltanen2015}, that includes realistic uncertainties arising from the tomography hardware.
In all cases, we present reconstruction results on a grid of $512 \times 512$ pixels from noisy modulo Radon projections
\begin{equation*}
\set{p_\delta^\lambda[m, n]}{m \in \idxset{0}{M-1}, ~ n \in \idxset{0}{N-1}}
\end{equation*}
of noise level $\delta > 0$ in the sense that $\|p^\lambda - p_\delta^\lambda\|_\infty \leq \delta$ and use the cosine filter with window function $W(S) = \cos(\frac{\pi S}{2}) \, \mathrm{1}_{[-1,1]}(S)$ and optimal bandwidth $L = M$.
We compare our novel LMU-FBP algorithm with the US-FBP method from~\cite{Beckmann2022}, which is based on Unlimited Sampling (US)~\cite{Bhandari2020}.
Note that US-FBP is designed for recovering bandlimited functions but it can be adapted to non-bandlimited data by manually setting the order of forward differences.
Here, we always choose order $1$ as higher orders are observed to fail in our examples.

\begin{figure}[t]
\centering
\includegraphics[width=\textwidth]{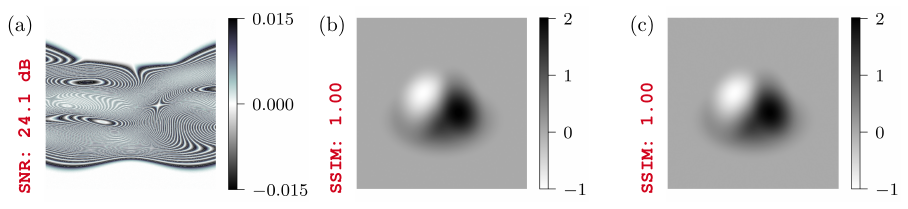}
\caption{Numerical experiments with smooth phantom. (a) Noisy modulo Radon data with $\lambda = 0.015$ and $\delta = 0.05\cdot\lambda$. (b) US-FBP on (a). (c) LMU-FBP on (a).}
\label{fig:smooth}
\end{figure}

\smallskip
\noindent
{\bf Smooth phantom.}
In our first set of proof-of-concept simulations, we consider the smooth phantom from~\cite{Rieder2003} with smoothness parameter $\nu = 2.5$ so that our assumptions of Theorem~\ref{theo:inversion_mrt} are satisfied.
Hence, we expect nearly perfect reconstruction via LMU-FBP.
The simulated modulo Radon data with $\lambda = 0.015$ is shown in Fig.~\ref{fig:smooth}(a), compressing the dynamic range by $50$ times and corrupted by uniform noise with noise level $\delta = 0.05 \cdot \lambda$ yielding a signal-to-noise ratio (SNR) of $24.1 \, \mathrm{dB}$.
We use the parameter choices $M = 360$ and $K = 1958$ leading to $N = 3917$ so that $T = \frac{1}{K} \leq \frac{1}{2L\e}$, which guarantees that US-FBP stably recovers an $L$-bandlimited function from MRT samples.
Although the smooth phantom is not bandlimited, we see in Fig.~\ref{fig:smooth}(b) that US-FBP nearly perfectly recovers with a structural similarity index measure (SSIM) of $1.00$.
The same is true for our newly proposed LMU-FBP reconstruction scheme, see Fig.~\ref{fig:smooth}(c).

\begin{figure}[t]
\centering
\includegraphics[width=\textwidth]{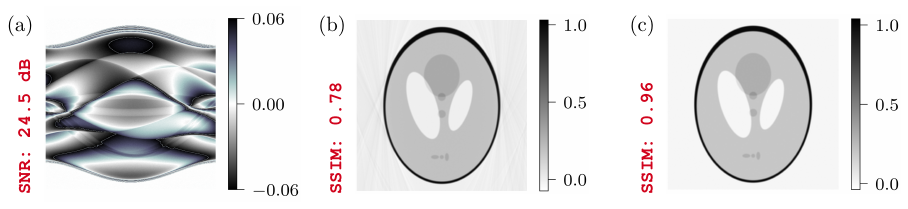}
\caption{Numerical experiments with Shepp-Logan phantom. (a) Noisy modulo Radon data with $\lambda = 0.06$ and $\nu = 0.05\cdot\lambda$. (b) US-FBP on (a). (c) LMU$_+$-FBP on (a).}
\label{fig:shepp-logan}
\end{figure}

\smallskip
\noindent
{\bf Shepp-Logan phantom.}
To also deal with a non-smooth test case, we now consider the classical Shepp-Logan phantom, which is piecewise constant and has jump discontinuities so that our assumptions in Theorem~\ref{theo:inversion_mrt} are {\em not} satisfied.
The simulated MRT data with $\lambda = 0.06$ is shown in Fig.~\ref{fig:shepp-logan}(a), compressing the dynamic range by about $5$ times and corrupted by uniform noise with $\delta = 0.05 \cdot \lambda$ leading to an SNR of $24.5 \, \mathrm{dB}$.
In this case, we see that US-FBP introduces artefacts in the reconstruction, cf.\ Fig.~\ref{fig:shepp-logan}(b), while our improved LMU$_+$-FBP method yields a nearly perfect reconstruction with SSIM of $0.96$, cf.\ Fig.~\ref{fig:shepp-logan}(c).

\begin{figure}[t]
\centering
\includegraphics[width=\textwidth]{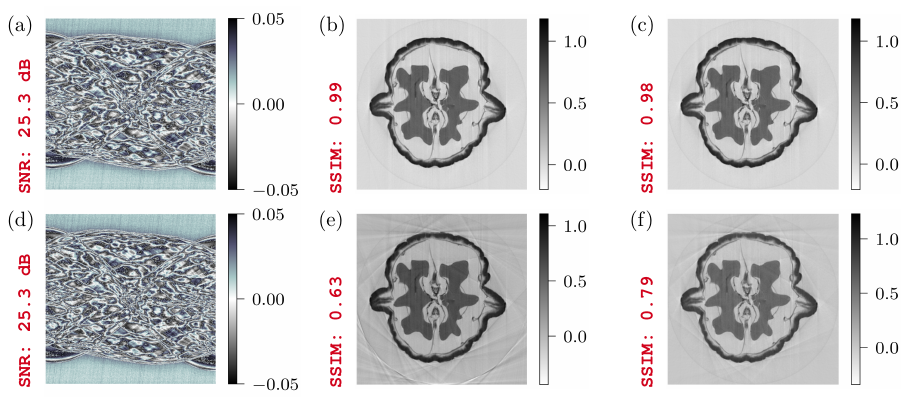}
\caption{Numerical experiments with walnut dataset. (a) Noisy modulo Radon data with $\lambda = 0.05$ and $\nu = 0.05\cdot\lambda$. (b) US-FBP on (a). (c) LMU-FBP on (a). (d) 2-times downsampled noisy modulo Radon data. (e) US-FBP on (d). (f) LMU-FBP on (d).}
\label{fig:walnut}
\end{figure}

\smallskip
\noindent
{\bf Walnut data.}
We finally present reconstruction results for the walnut dataset from~\cite{Siltanen2015}, which is transformed to parallel beam geometry with $M = 600$ and $K = 1128$.
Moreover, the Radon data is normalized to the dynamical range $[0,1]$, see Fig.~\ref{fig:phantoms}(c).
The corresponding FBP reconstruction is shown in Fig.~\ref{fig:phantoms}(f) and serves as ground truth for comparing our reconstruction results.
Simulated modulo Radon projections with $\lambda = 0.05$ are displayed in Fig.~\ref{fig:walnut}(a), where we added uniform noise with $\delta = 0.05 \cdot \lambda$ to account for quantization errors leading to an SNR of $25.3 \, \mathrm{dB}$.
The reconstruction with US-FBP is shown in Fig.~\ref{fig:walnut}(b) and with LMU-FBP in Fig.~\ref{fig:walnut}(c).
Both algorithm yield a reconstruction of the walnut that is indistinguishable from the FBP reconstruction with an $\SSIM$ of $0.99$ and $0.98$, respectively, while compressing the dynamic range by $10$ times.
The results for twice radially downsampled noisy MRT data are shown in Fig.~\ref{fig:walnut}(d)-(f).
While US-FBP produces severe artefacts, LMU-FBP still gives a decent reconstruction.

\section{Conclusion}

In this work, we proved a novel analytical inversion formula for the MRT closing a gap in the existing literature.
Discretization with Fourier techniques lead to the new LMU-FBP algorithm, which can handle non-bandlimited Radon data and performs on par or even better than US-FBP in this case.
Future work includes weakening the assumptions and analyzing recovery guarantees for discrete data.

%
%
\begin{credits}
\subsubsection{\ackname}
This work was supported by the Deutsche Forschungsgemeinschaft (DFG) - Project numbers 530863002 and 281474342/GRK2224/2.
\end{credits}

%
%
\bibliographystyle{splncs04}

\end{document}